%% file: main.tex
\documentclass[a4paper,12pt,reqno]{amsart}  %%%%% When compiling with pdfLaTeX

\usepackage{amsmath}
\usepackage[all]{xy}
\usepackage{tikz-cd}
\usepackage{amssymb}
\usepackage{amsthm}
\usepackage{amsfonts}
\usepackage{mathtools}
\usepackage{comment}
\usepackage{mathrsfs}
\usepackage{pifont}
\usepackage{cite}
\usepackage{enumerate}
\usepackage{here}
\usepackage{autobreak}
\usepackage{wrapfig}
\usepackage{graphicx}
\usepackage{lscape}
\usepackage[colorlinks]{hyperref}  %%%%%%%%%% pdfLaTeX ver
\usepackage{xcolor}
\hypersetup{
	bookmarksnumbered=true,
    colorlinks=true,
    citecolor=blue,
    linkcolor=purple,
    urlcolor=orange,
}

\usepackage{pgfplots}
\pgfplotsset{compat=1.18}

\usepackage{enumitem} %% [label={(M\arabic*)}] の使用

\usepackage[capitalize, nameinlink]{cleveref}
\everymath{\displaystyle}
\numberwithin{equation}{section}

\newcommand{\redtext}[1]{\textcolor{red}{#1}}
\newcommand{\memo}[1]{\redtext{[{#1}]}} %% \memo{#1}

\theoremstyle{plain}

\newtheorem{thm}{Theorem}[section]
\crefname{thm}{Theorem}{Theorems}

\newtheorem{lem}[thm]{Lemma}
\crefname{lem}{Lemma}{Lemmas}
\newtheorem{prop}[thm]{Proposition}
\crefname{prop}{Proposition}{Propositions}
\newtheorem{cor}[thm]{Corollary}
\crefname{cor}{Corollary}{Corollaries}
\newtheorem{claim}{Claim}[section]

\newtheorem{introthm}{Theorem}[section]
\crefname{introthm}{Theorem}{Theorems}

\newtheorem{introcor}[introthm]{Corollary}

\theoremstyle{definition}

\newtheorem{dfn}[thm]{Definition}

\newtheorem{eg}[thm]{Example}
\newtheorem{setup}[thm]{Setup}
\newtheorem*{Ack}{Acknowledgement}
\newtheorem*{NoCon}{Notation and Conventions}
\newtheorem*{Out}{Outline of this paper}

\theoremstyle{remark}
\newtheorem*{rem}{Remark}

\DeclareMathOperator{\Ext}{Ext}
\DeclareMathOperator{\ext}{ext}
\DeclareMathOperator{\Pic}{Pic}
\DeclareMathOperator{\Aut}{Aut}
\DeclareMathOperator{\Auteq}{Auteq}

\DeclareMathOperator{\Hilb}{Hilb}
\DeclareMathOperator{\Sym}{Sym}

\DeclareMathOperator{\NS}{NS}

\DeclareMathOperator{\ch}{ch}

\DeclareMathOperator{\id}{id}

\DeclareMathOperator{\Cone}{Cone}

\DeclareMathOperator{\Deck}{Deck}

\newcommand\Hom{\mathop{\mathrm{Hom}}\nolimits}
\newcommand\RHom{\mathop{\mathbf{R}\mathrm{Hom}}\nolimits}
\newcommand{\PP}[1]{\mathbb{P}^{#1}}
\newcommand{\thcl}[1]{\langle{#1}\rangle_{\mathrm{th}}}
\newcommand{\abs}[1]{\lvert{#1}\rvert}
\newcommand{\lderived}{\mathbf{L}}
\newcommand{\rderived}{\mathbf{R}}
\newcommand{\linsys}[1]{\lvert{#1}\rvert}

\DeclareMathOperator{\For}{\mathsf{For}}
\DeclareMathOperator{\Ind}{\mathsf{Ind}}
\DeclareMathOperator{\BKR}{\mathsf{BKR}}

\newcommand\dual{\raise0.9ex\hbox{$\scriptscriptstyle\vee$}}

\newcommand{\htop}{h_{\mathrm{top}}}
\newcommand{\hcat}{h_{\mathrm{cat}}}
\newcommand{\numg}[1]{{#1}^{N}}
\newcommand{\Phin}{\Phi^{\boxtimes n}}
\newcommand{\boxn}{\boxtimes n}

\newcommand{\CB}{\mathbb{C}}

\newcommand{\RB}{\mathbb{R}}

\newcommand{\ZZ}{\mathbb{Z}}

\newcommand{\DC}{\mathcal{D}}

\newcommand{\LC}{\mathcal{L}}

\newcommand{\OO}{\mathscr{O}}

\author{Tomoki Yoshida}
\address[T.Yoshida]{Department~of~Mathematics, School~of~Science~and~Engineering, Waseda~University, Ohkubo~3-4-1, Shinjuku, Tokyo~169-8555, Japan}
\email{\href{mailto:tomoki_y@asagi.waseda.jp}{tomoki\_y@asagi.waseda.jp}}

\title[Categorical Entropies of Hilbert Schemes and Hyperk\"ahler Manifolds]{Categorical Entropies of Hilbert Schemes of points on Surfaces and Hyperk\"ahler Manifolds}
\date{August 10, 2026}

\begin{document}
\begin{abstract}
This paper studies the categorical entropy of autoequivalences of derived categories of Hilbert schemes of points on surfaces and hyperk\"ahler manifolds.
One of the central questions about categorical entropy is whether it satisfies a Gromov–Yomdin type formula $h_{\mathrm{cat}}(\Phi) = \log\rho(\Phi)$. We say that $X$ has the Gromov–Yomdin (GY) property if this formula holds.
We prove that if a surface $S$ with $q(S)=0$ fails to satisfy the (GY) property (e.g., K3 surfaces), then so does $\mathrm{Hilb}^n(S)$.
We further prove that no projective hyperk\"ahler or Enriques manifold satisfies the (GY) property by constructing an explicit autoequivalence with positive categorical entropy but unipotent action on the cohomology ring. 
In fact, this discrepancy is arbitrarily large: categorical entropy is
unbounded among autoequivalences with spectral radius one.
\end{abstract}
\maketitle
\setcounter{tocdepth}{2} %table of contents subsection
\tableofcontents
\setcounter{section}{-1}
\input{introduction.tex}
\input{contents}
\bibliographystyle{amsalpha}
\bibliography{bibtex_tyoshida}
\end{document}

%% file: introduction.tex
\section{Introduction}\label{section: introduction}
A dynamical system $(X, f)$ is a pair of a topological space and an endomorphism.
To quantify the ``complexity'' of an automorphism of infinite order, the \emph{entropy} of the dynamical system plays an important role.
The entropy $\htop(f)$ of $(X, f)$ is a real number defined by using a limit of iterated actions, although the actual definition depends on the context.
It is often difficult to determine or calculate its value.

We consider only the case where $X$ is a smooth projective variety over $\CB$. 
In such a case, the following formula holds:

\begin{thm}[\cite{gromov_1987_entropy_homology_and_semialgebraic_geometry,gromov_2003_on_the_entropy_of_holomorphic_maps,yomdin_1987_volume_growth_and_entropy}]
\label{theorem: Gromov-Yomdin equality}
    Let $X$ be a smooth projective variety over $\CB$ and $f$ be a surjective endomorphism of $X$.
    Then, the following equality holds:
    \[
        \htop(f) = \log\rho(f^*\rvert_{\oplus_{p=0}^{\dim X}H^{p, p}(X, \ZZ)}), 
    \]
    where $\rho(-)$ is the spectral radius of the induced linear map on $\oplus_{p}H^{p,p}(X, \ZZ)$.
\end{thm}

However, the cases in which the entropy is actually positive are relatively rare.
For the case where $X$ is a surface, Cantat \cite{cantat_1999_dynamique_des_automorphismes_des_surfaces_projectives_complexes} showed that if $X$ admits an automorphism of positive topological entropy, $X$ is birational to the projective plane $\PP{2}$, a K3 surface, an Enriques surface, or a complex torus.
While no full classification analogous to Cantat's is known in higher dimensions, identifying the class of varieties that can admit positive entropy remains a major open problem.

\subsection{Dynamical System on Derived Categories}\label{subsection: dynamical system on der cats}
Let $X$ be a smooth projective variety over $\CB$ and $D^b(X)$ be the derived category of coherent sheaves on $X$.
The group of autoequivalences of $D^b(X)$ contains information about $\Aut(X)$ as part of the standard autoequivalences.
That is, if we denote
\[
A(X) \coloneqq \Pic(X) \rtimes \Aut(X) \times\ZZ[1], 
\]
there is an injection $A(X)\hookrightarrow \Auteq(X)$, whose image is called the group of \emph{standard autoequivalences}. Note that if $\omega_X$ is ample or anti-ample, $\Auteq(X) = A(X)$ due to Bondal and Orlov \cite{bondal_orlov_2001_reconstruction_of_a_variety_from_the_derived_category_and_groups_of_autoequivalences}.
In this context, it is natural to investigate dynamical systems on triangulated categories.
In \cite{dimitrov_haiden_katzarkov_kontsevich_2014_dynamical_systems_and_categories}, Dimitrov, Haiden, Katzarkov, and Kontsevich generalized the notion of topological entropy for an automorphism to an autoequivalence of the derived category of coherent sheaves.

However, in general, it is difficult to determine the value of $\hcat(\Phi)$ for a given $\Phi\in \Auteq X$. An autoequivalence of $D^b(X)$ induces automorphisms of $H^{*}(X, \CB)$ and of the numerical Grothendieck group $N(X)_{\CB}$. We denote the latter by $\numg{\Phi}$.
It is desirable to have a formula for calculation in comparison with the case of topological entropy, i.e., for every $\Phi\in\Auteq X$, 
\[
\hcat(\Phi) = \log\rho(\numg{\Phi}), 
\]
as an analogue of \cref{theorem: Gromov-Yomdin equality}.
As is mentioned later, it is proved that this equality holds for some specific classes of varieties, while it does not hold for other classes of varieties.
We say that $X$ has the \emph{Gromov-Yomdin (GY) property} if the following condition holds:

\begin{enumerate}[label={\textbf{(GY)}}, ref={(GY)}]
    \item \label{property: Gromov-Yomdin property}
\begin{center}
    for all $\Phi\in\Auteq (X)$, $\hcat(\Phi) = \log\rho (\numg{\Phi})$. 
\end{center}
\end{enumerate}

Motivated by Cantat's result and Gromov-Yomdin equality (\cref{theorem: Gromov-Yomdin equality}), common to studies of categorical entropies, the main questions concerning categorical entropy are the following:
\begin{enumerate}
    \item Which classes admit autoequivalences of positive categorical entropy?
    \item Which classes have the \ref{property: Gromov-Yomdin property} property?
\end{enumerate}

Previous works have provided examples that satisfy or fail to satisfy the (GY) property.
Kikuta was the first to discuss these problems and showed that (GY) holds for curves \cite{kikuta_2017_on_entropy_for_autoequivalences_of_the_derived_category_of_curves}.
Then, Kikuta and Takahashi proved the Gromov-Yomdin type equality for all standard autoequivalences of derived categories of coherent sheaves on any smooth projective variety over $\CB$ \cite{kikuta_takahashi_2019_on_the_categorical_entropy_and_the_topological_entropy}.
In particular, if $X$ has an ample or an anti-ample canonical divisor, $X$ has the (GY) property. 
Also, Kikuta, Shiraishi, and Takahashi proved the same result for orbifold projective lines \cite{kikuta_shiraishi_takahashi_2020_a_note_on_entropy_of_autoequivalences_lower_bound_and_the_case_of_orbifold_projective_lines}, Yoshioka proved it for abelian surfaces and simple abelian varieties \cite{yoshioka_2020_categorical_entropy_for_fouriermukai_transforms_on_generic_abelian_surfaces}, and the author for bielliptic surfaces \cite{yoshida_2026_a_note_on_categorical_entropy_of_bielliptic_surfaces_and_enriques_surfaces}.

On the other hand, some examples are known not to satisfy the (GY) property.
Fan observed that every even-dimensional Calabi-Yau hypersurface with $\dim >3$ fails to satisfy the (GY) property \cite{fan_2018_entropy_of_an_autoequivalence_on_calabiyau_manifolds}.
After that, Ouchi discovered that no projective K3 surface has the (GY) property in \cite{ouchi_2020_on_entropy_of_spherical_twists}. 
This result was extended by Mattei \cite{mattei_2021_categorical_vs_topological_entropy_of_autoequivalences_of_surfaces} to the case of surfaces containing a $(-2)$-curve.
In \cite{barbacovi_kim_2023_entropy_of_the_composition_of_two_spherical_twists}, Barbacovi and Kim showed that, for even integers $m, n\ge2$, a divisor of type $(m+1, n+1)$ in $\PP{m}\times\PP{n}$ also fails to satisfy (GY) \cite[Example 5.0.2.]{barbacovi_kim_2023_entropy_of_the_composition_of_two_spherical_twists}
by determining the categorical entropy of the composition of two spherical functors.
Also, the author 
\cite{yoshida_2026_a_note_on_categorical_entropy_of_bielliptic_surfaces_and_enriques_surfaces}
showed that Enriques surfaces do not satisfy (GY) by using their canonical covering.

\subsection{Results}\label{subsection: Results}
The first aim of this paper is to investigate categorical dynamics on derived categories of Hilbert schemes of points on surfaces, especially those that admit an autoequivalence of positive categorical entropy. Our first theorem is as follows:
\begin{introthm}[\cref{theorem: positivity inherits to hilbert scheme,theorem: GY property for Hilb}]\label{theorem in intro: GY property for Hilb}
    Let $S$ be a smooth projective surface.
    Assume that there is an autoequivalence $\Phi\in\Auteq(S)$ such that $\hcat(\Phi)>0$.
    Then, $\Hilb^{n}(S)$ also admits an autoequivalence with positive categorical entropy for every $n\in\ZZ_{>0}$.
    
    Moreover, assume that $S$ does not satisfy \ref{property: Gromov-Yomdin property} and $q(S)=0$.
    Then, $\Hilb^n(S)$ fails to satisfy the \ref{property: Gromov-Yomdin property} property for every $n\in\ZZ_{>0}$.
\end{introthm}

In particular, considering the case where $S$ is a K3 surface, the theorem gives an example of hyperk\"ahler manifolds not having (GY).

The key ingredients of the proof are \emph{equivariant categories} and the so-called BKR-Haiman equivalences. 
Let $\Phi$ be an autoequivalence of the derived category of a surface $S$ and $\Psi$ be the corresponding autoequivalence on $D^b(\Hilb^{n}(S))$.
Then, the direct calculation of the categorical entropy and spectral radius of $\Psi$ implies the assertion.

The next interest of this paper is \emph{hyperk\"ahler manifolds}. Let us review the definition:
\begin{dfn}\label{definition: hyperkahler manifold}
    A (compact) hyperk\"ahler manifold $X$ is a simply connected complex K\"ahler manifold admitting an everywhere non-degenerate global holomorphic $2$-form $\omega$ such that $H^{2, 0}(X, \CB) = \CB\omega$.
\end{dfn}
As one of the three basic blocks appearing in the Beauville–Bogomolov decomposition, hyperk\"ahler manifolds are a natural subject of study. In particular, the group of autoequivalences of their derived categories is known to possess symmetries that are not found in the other types while hyperk\"ahler manifolds are conjectured not to admit spherical objects; this conjecture has been verified for $\text{K}3^{[n]}$-type ($n>1$) or OG$10$-type \cite[Theorem 1.9. and Conjecture A.7.]{beckmann_2025_atomic_objects_on_hyperkahler_manifolds}.
A typical non-standard autoequivalence on hyperk\"ahler manifolds is the \emph{$\PP{n}$-twist}.
The entropy of $\PP{n}$-twist $P$ has already been calculated in \cite[Theorem 2.13.]{preprint_yuwei_2025_on_entropy_of_mathbbptwists}:
\[
\hcat(P) = 0 = \log\rho(\numg{P}).
\]

Despite this vanishing of entropy, we show that $\PP{n}$-twist produces autoequivalences with positive categorical entropy and unipotent action on the cohomology ring.
Our second theorem is as follows:
\begin{introthm}[\cref{theorem: GY property of hyperkahler manifold}]\label{theorem in intro: GY property for HK}
    Let $X$ be a projective hyperk\"ahler manifold.
    Then, $D^b(X)$ admits an autoequivalence with positive categorical entropy.
    Moreover, no hyperk\"ahler manifold $X$ satisfies the \ref{property: Gromov-Yomdin property} property.
\end{introthm}

We will see that the autoequivalence in the theorem is given by the composition of the $\PP{n}$-twist associated to a $\PP{n}$-object $\OO_X$ and tensoring with the inverse of a very ample line bundle.

\begin{rem}
     We recall that for a hyperkähler manifold $X$ with small Picard number, the automorphism group $\Aut(X)$ is often dynamically trivial, meaning that $\htop(f) = 0$ for all $f \in \Aut(X)$ (see \cite{oguiso_2007_automorphisms_of_hyperkahler_manifolds_in_the_view_of_topological_entropy} for instance). 
    While it is already well-known that the group of autoequivalences is larger than $\mathrm{Aut}(X)$, our result implies that this extension is non-trivial not only algebraically but also dynamically.
    More precisely, although $\htop(f)=0$ for every $f\in\Aut(X)$, there exist $\Phi\in\mathrm{Auteq}(X)$ with $h_{\mathrm{cat}}(\Phi)>0$.  Hence the enlargement $\Aut(X)\subset\mathrm{Auteq}(X)$ yields genuinely new dynamical phenomena.

    Regarding the (GY) property, as it is already known to hold in some cases for Abelian varieties and to fail for certain strict Calabi-Yau manifolds, hyperk\"ahler manifolds have been the last missing block in the Beauville-Bogomolov decomposition. 
    The theorem fills this gap by providing the first counterexamples in the hyperk\"ahler setting.
\end{rem}

In the previous work of the author \cite{yoshida_2026_a_note_on_categorical_entropy_of_bielliptic_surfaces_and_enriques_surfaces}, we showed that no Enriques surface satisfies (GY) by applying a descent argument to the covering K3 surface.
We can extend the argument to the case of \emph{Enriques manifolds}, introduced by Oguiso and Schr\"oer \cite{oguiso_schroer_2011_enriques_manifolds} (see \cref{definition: Enriques manifolds} for the precise definition).
\begin{introthm}[\cref{theorem: pos entr and counteg of GY for Enriques manifolds}]\label{theorem in intro: GY property for Enriques manifolds}
    Let $Y$ be an Enriques manifold.
    Then, $D^b(Y)$ has an autoequivalence with positive categorical entropy.
    Moreover, no Enriques manifold $Y$ satisfies the \ref{property: Gromov-Yomdin property} property.
\end{introthm}

By the argument in \cref{theorem in intro: GY property for HK,theorem in intro: GY property for Enriques manifolds}, we have the following unboundedness results:

\begin{introcor}[\cref{corollary: hyperkaheler the set of hcat is unbounded,corollary: Enriques the set of hcat is unbounded}]
\label{corollary in intro: hyperkahler Enriques the set of hcat is unbounded}
Let $X$ be a projective hyperk\"ahler manifold or an Enriques manifold.
Then, for any $M>0$, there exists an autoequivalence $\Phi\in\Auteq X$ such that $\rho(\numg{\Phi}) = 1$ and $\hcat(\Phi)>M$. Equivalently, we have
\[
\sup\{\hcat(\Phi) \mid \Phi\in\Auteq X, \rho(\numg{\Phi}) = 1\} =\infty.
\]
\end{introcor}
\begin{rem}
The point of \cref{corollary in intro: hyperkahler Enriques the set of hcat is unbounded} is that the discrepancy in the Gromov--Yomdin equality can be arbitrarily large. 
More precisely, even when the numerical action has spectral radius one, the categorical entropy can be arbitrarily large. Thus, numerical dynamics may remain simple from the viewpoint of spectral radius, while the corresponding categorical dynamics becomes arbitrarily complicated.
\end{rem}
\begin{Out}
\cref{section: preliminaries cat entropy,section: non-standard auto of der cats} recall the definitions and preliminary results needed for the rest of the paper, with particular emphasis on categorical entropy and on autoequivalences of derived categories of coherent sheaves.
In \cref{section: categorical entropies on hilb}, we prove \cref{theorem in intro: GY property for Hilb}.
For this purpose, we use the framework of equivariant categories. \cref{subsection: equivariant categories} is devoted to a review of equivariant categories, and the proof of \cref{theorem in intro: GY property for Hilb} is given in \cref{subsection: hilb^n cat entro}.

\cref{theorem in intro: GY property for HK} will be shown in \cref{section: categorical entropy on hyperkahler manifold}. We give an explicit example that violates the (GY) property for any hyperk\"ahler manifold. 
We prove this by obtaining a lower bound for the categorical entropy in that example via computations of the dimensions of the cohomologies.
As an application of this, the following \cref{section: Enriques manifolds} shows \cref{theorem in intro: GY property for Enriques manifolds} by using the technique in \cite{yoshida_2026_a_note_on_categorical_entropy_of_bielliptic_surfaces_and_enriques_surfaces}  that relates categorical entropy on quotients to that on its canonical covering.
\end{Out}

\begin{NoCon}
The conventions and notations used in this paper are listed below:
\begin{itemize}
    \item all varieties are smooth projective varieties defined over $\CB$, 
    \item $h^i(E) \coloneqq \dim H^i(E)$ for $E\in D^b(X)$,
    \item $\ext^i(E, F)\coloneqq \dim\Ext^{i}(E, F)  = \dim\Hom(E, F[i]) =\hom^i(E, F)$ for $E, F\in D^b(X)$, and
    \item all functors are derived unless otherwise stated. However, the symbols $\lderived$ and $\rderived$ are occasionally used to emphasize that the functor is actually derived.
\end{itemize}
\end{NoCon}

\begin{Ack}
    I would like to thank Professors Yasunari Nagai and Genki Ouchi and Dr. Hayato Arai for their useful comments and conversations.
\end{Ack}

%% file: contents.tex
\section{Entropies of Exact Endofunctors}\label{section: preliminaries cat entropy}

To see the definition of categorical entropy, we review the notion of classical generator of a triangulated category.

\begin{dfn}\label{definition: thick closure and classical generator}
    Let $\DC$ be a triangulated category and $G$ an object in $\DC$.
    The \emph{thick closure} of $G$ is the smallest triangulated subcategory that contains $G$ and is closed under taking direct summands.
    We denote it by $\thcl{G}$.

    An object $G$ is called a \emph{classical generator} if $\thcl{G} = \DC$ holds.
\end{dfn}

    As shown by Orlov, a generator always exists in our setting:

\begin{thm}[\cite{orlov_2009_remarks_on_generators_and_dimensions_of_triangulated_categories}]
\label{theorem: existence of generator}
    Let $\OO(1)$ be a very ample line bundle on a smooth projective variety $X$.
    Then, the object
    \[
    \bigoplus_{i=0}^{\dim X}\OO(i+k)
    \]
    is a classical generator of $D^b(X)$ for any $k\in\ZZ$.
\end{thm}

Next, we review the definition of \emph{categorical entropy} of an exact functor by \cite{dimitrov_haiden_katzarkov_kontsevich_2014_dynamical_systems_and_categories}.

\begin{dfn}\label{definition: entropy}
    Let $\DC$ be a triangulated category, $G$ and $G'$ generators of $\DC$, $\Phi$ an exact endofunctor of $\DC$, $E\in\DC$, and $t\in\RB$.
    The \emph{complexity $\delta_t(E, G)$ of $E$ (with respect to $G$)} is defined to be
    \[
        \delta_t(E, G)\coloneqq \inf\left\{ \sum_{i=1}^k e^{-n_it}\ \  \middle|\ \ 
\begin{xy}
(0,5) *{0}="0", (20,5)*{E_{1}}="1", (30,5)*{\dots}, (40,5)*{E_{k-1}}="k-1", (60,5)*{E\oplus E'}="k",
(10,-5)*{G[n_{1}]}="n1", (50,-5)*{G[n_{k}]}="nk",
\ar "0"; "1"
\ar "1"; "n1"
\ar@{.>} "n1";"0"
\ar "k-1"; "k" 
\ar "k"; "nk"
\ar@{.>} "nk";"k-1"
\end{xy}
\begin{comment}
\begin{tikzcd}
0=E_0 \arrow[rr] & & E_1 \arrow[ld] \arrow[r] & \cdots \arrow[r] & E_k \arrow[rr] & & E\oplus E' \arrow[ld] \\
& {G[n_1]} \arrow[lu, dotted] & & & & {G[n_k]} \arrow[lu, dotted] &
\end{tikzcd}
\end{comment}
        %\sum_{i=1}^k e^{n_it} \mid E_{i-1} \to E_i \to G[n_i] \text{ for }1\le i\le k, \text{ where } E_0 =0, E_k=E\oplus E'
        \right\}.
    \]

    The \emph{entropy of $\Phi$} is defined to be 
    \[
    h_t(\Phi) \coloneqq \lim_{n\to \infty} \frac{1}{n}\log\delta_t(G, \Phi^n(G')).
    \]
    Moreover, $\hcat(\Phi)\coloneqq h_0(\Phi)$ is called the \emph{categorical entropy} of $\Phi$.
\end{dfn}

\begin{thm}[\cite{dimitrov_haiden_katzarkov_kontsevich_2014_dynamical_systems_and_categories}]\label{theorem: categorical entropy independent and finite}
    With the setting above, $h_t(\Phi) \in[-\infty, \infty )$, and the entropy is independent of the choice of generators $G$ and $G'$. Moreover, $\hcat(\Phi)\in [0, \infty)$.
\end{thm}

\begin{thm}\label{theorem: simple form of entropy for ext finite}
    Assume that $\DC$ is an $\Ext$-finite triangulated category.
    Let $E, F\in\DC$, $t\in\RB$, and
    \[
    \delta'(E, F) \coloneqq \sum_{k\in\ZZ}\ext^k(E, F)e^{-kt}.
    \]
    Then, for an exact endofunctor $\Phi$, the following holds:
    \[
    h_t(\Phi) = \lim_{n\to \infty} \frac{1}{n}\log \delta'(G, \Phi^n(G')), 
    \]
    where $G$ and $G'$ are generators of $\DC$.
\end{thm}

For an exact functor $\Phi: D^b(X) \to D^b(X)$, we denote the induced linear map on the numerical Grothendieck group by $\numg{\Phi}: N(X)_{\CB} \to N(X)_{\CB}$.

\section{Non-Standard Autoequivalences of Derived Categories}\label{section: non-standard auto of der cats}

\begin{dfn}
    Let $E$ be an object in a triangulated category $\DC$ with the Serre functor $S: \DC\to \DC$.
    The object $E$ is called \emph{$d$-spherical} if $S(E)\cong E[d]$ and

    \begin{align*}
        \hom^i(E, E) = 
        \begin{cases}
            \CB & i=0, d    \\
            0 & \text{otherwise}.
        \end{cases}
    \end{align*}
    For the derived category $D^b(X)$ of a $d$-dimensional variety $X$, $d$-spherical object $E\in D^b(X)$ is simply called spherical.

    For a spherical object $E\in \DC$, it induces an autoequivalence, which is called \emph{spherical twist} (\cite{seidel_thomas_2001_braid_group_actions_on_derived_categories_of_coherent_sheaves}) and is denoted by $T_E$.
    We only review the correspondence of objects here:
    \[
        T_E: F \mapsto \Cone\big(\RHom(E, F)\otimes E \stackrel{\mathrm{ev}}{\longrightarrow} F\big).
    \]
 \end{dfn}

\begin{eg}
Some basic examples of spherical objects are as follows:
\begin{itemize}
    \item The structure sheaf $\OO_S$ on a K3 surface $S$ is a spherical object.
    \item Let $C$ be a $(-2)$-curve on a smooth projective surface $S$. The torsion sheaf $\OO_C$ is a spherical object in $D^b(S)$.
    \item For a (strict) Calabi-Yau manifold, the structure sheaf is spherical.
\end{itemize}
\end{eg}

\begin{dfn}
    Let $F$ be an object in a triangulated category $\DC$ with the Serre functor $S: \DC\to \DC$.
    The object $F$ is called a $\PP{n}$\emph{-object} if $S(F)\cong F[2n]$ and

    \begin{align*}
        \hom^i(F, F) = 
        \begin{cases}
            \CB & i=0, 2, \ldots, 2n-2, 2n    \\
            0 & \text{otherwise}.
        \end{cases}
    \end{align*}

    A $\PP{n}$-object $F\in\DC$ induces an autoequivalence, which is called \emph{$\PP{n}$-twist} and denoted by $P_F$. 
    Let $h\in\Hom(F[-2], F) (\cong\Hom(F, F[2]))$ be a generator and $h\dual\in\Hom(F\dual, F\dual[2])(\cong\Hom(F, F[2]))$.
    
    The correspondence of objects is given by a double cone as follows (\cite{huybrechts_thomas_2006_bbb_pobjects_and_autoequivalences_of_derived_categories}):
    \[
        P_F(G) = \Cone\bigg( \Cone\big( \Hom^{*-2}(F,G)\otimes F \stackrel{h\dual\cdot\id - \id\cdot h}{\longrightarrow} \Hom^{*}(F, G)\otimes F\big) \to G \bigg).
    \]
\end{dfn}
    For explicit constructions of these as Fourier-Mukai functors, see \cite[Section 8]{book_huybrechts_2006_fouriermukai_transforms_in_algebraic_geometry}.
\begin{eg}
The notion of $\PP{n}$-objects arises from the context of hyperk\"ahler manifolds. 
\begin{itemize}
    \item Let $X$ be a hyperk\"ahler manifold of dimension $2n$.
    The structure sheaf $\OO_X$ is a $\PP{n}$-object.

    \item Let $L\cong\PP{n}$ be a Lagrangian submanifold of a compact hyperk\"ahler manifold $X$ of dimension $2n$.
    Then, $\OO_{L}$ is a $\PP{n}$-object of $D^b(X)$.
\end{itemize}
\end{eg}

\begin{thm}[{\cite{ouchi_2020_on_entropy_of_spherical_twists}}]
    Let $S$ be a K3 surface and $H$ a very ample line bundle on $S$ such that $H^2\ge 10$.
    Then, the autoequivalence 
    \[
    \Phi \coloneqq T_{\OO_S} \circ (- \otimes \OO_S(-H))
    \]
    satisfies the inequality
    \begin{equation*}
        \hcat(\Phi) > \log\rho(\numg{\Phi})\ge0.
    \end{equation*}
    In particular, no K3 surface satisfies the property (GY).

\end{thm}

On the other hand, D. Mattei showed that a surface admitting a $(-2)$-curve does not satisfy the property (GY).
The precise statement is the following:

\begin{thm}[{\cite[Theorem 1.2.]{mattei_2021_categorical_vs_topological_entropy_of_autoequivalences_of_surfaces}}]
Let $S$ be a smooth projective surface, $C$ a $(-2)$-curve on $S$, and $\LC$ a line bundle on $S$.
Assume that $\deg_C(\LC\rvert_C)<0$. 
Then, the autoequivalence $\Phi\coloneqq T_{\OO_C}\circ (-\otimes \LC)$ satisfies the following inequality:
\[
    \hcat(\Phi) > \log\rho(\numg{\Phi}).
\]
\end{thm}

Also, the author showed it for Enriques surfaces:

\begin{thm}[\cite{yoshida_2026_a_note_on_categorical_entropy_of_bielliptic_surfaces_and_enriques_surfaces}]\label{theorem: Enriques surface GY property}
    Let $S$ be an Enriques surface and $\pi: X \to S$ its canonical covering.
    Ouchi's autoequivalence on $X$ descends to $D^b(S)$ and it does not satisfy the Gromov-Yomdin type equality.
    In particular, every Enriques surface admits an autoequivalence of positive categorical entropy and does not satisfy the property (GY).
\end{thm}

\section{Categorical Entropies on Hilbert Schemes of points on surfaces}\label{section: categorical entropies on hilb}
\subsection{Equivariant Categories}\label{subsection: equivariant categories}
We review the definitions of equivariant categories and their autoequivalences. 
For more details, see \cite{beckmann_oberdieck_2023_on_equivariant_derived_categories,preprint_alexey_2015_on_equivariant_triangulated_categories,ploog_2007_equivariant_autoequivalences_for_finite_group_actions} for instance.

\begin{dfn}\label{definition: G action to tri cat}
    Let $\DC$ be a triangulated category and $G$ a finite group.
    An action of $G$ on $\DC$ consists of the data $(\{\rho_g\}_{g\in G}, \{\theta_{g,h}\}_{g,h\in G})$, where $\rho_g$ is an autoequivalence of $\DC$ and $\theta_{g,h}$ is an isomorphism of functors $\theta_{g,h}: \rho_g\circ\rho_h \to \rho_{gh}$ such that, for every $g,h,i\in G$, 
    \[
    \theta_{g,hi}\circ\rho_g\theta_{h,i} = \theta_{gh, i}\circ\theta_{g,h}\rho_i.
    \]
\end{dfn}

\begin{dfn}\label{definition: equivariant category}
    Let $G\curvearrowright \DC$ be the triangulated category $\DC$ with a group action $(\rho, \theta)$ of a finite group $G$.

    The \emph{equivariant category} $\DC_G$ is defined as follows:
    \begin{itemize}
        \item an object is a pair $(E, \phi)$ of $E\in\DC$ and $\phi= \{\phi_g: E\to\rho_gE\}_{g\in G}$ such that, for every $g,h\in G$, 
        \[
        \theta_{g,h}(E)\circ\rho_g\phi_h\circ\phi_g = \phi_{gh}, 
        \]
        and
        \item a morphism $f: (E, \phi) \to (E', \phi')$ is a morphism $f\in\Hom_\DC(E, E')$ such that, for every $g\in G$, 
        \[
        \phi'_g\circ f = \rho_g f\circ\phi_g.
        \]
    \end{itemize}
\end{dfn}

Let $\DC$ be a triangulated category and $G$ a finite group with an action $G\curvearrowright \DC$.
Denote $\DC_G$ the triangulated category of $G$-equivariant objects.

\begin{dfn}
    Let $G$ be a finite group and $\DC$ a triangulated category with an action $G\curvearrowright\DC$.
    There are natural functors between $\DC$ and the equivariant category, namely, the \emph{forgetful functor} and the \emph{induction functor}:
    \begin{align*}
    \For&: \DC_G \to \DC \quad (E, \phi) \mapsto E, \\
    \Ind&: \DC \to \DC_G \quad E \mapsto \bigg(\bigoplus_{g\in G} \rho_gE, \phi = \{\phi_g\}_{g\in G}\bigg), 
    \end{align*}
    where $\phi_g : \bigoplus_{h\in G}\rho_hE \to \bigoplus_{h\in G}\rho_g\rho_hE$
    is the direct sum of isomorphisms $\theta_{g, h}^{-1}: \rho_{gh}E\to \rho_g\rho_h E$ that are given in the data of the action $G\curvearrowright\DC$.
\end{dfn}
\begin{prop}[{\cite[Lemma 3.8.]{preprint_alexey_2015_on_equivariant_triangulated_categories}}]\label{proposition: adjoint of for and inf}
    The functor $\For$ is left and right adjoint to the functor $\Ind$.
\end{prop}

\begin{lem}\label{lemma: For Inf preserves the classical generator}
    The functors $\Ind$ and $\For$ preserve classical generators.
\end{lem}
\begin{proof}
    We show the claim for $\Ind$. 
    Let $G$ be a classical generator of $\DC$ and $(F, \rho)$ an object in $\DC_G$.
    Then, by assumption, $\For(F)\in \DC = \thcl{G}$.
    As the exact functor $\Ind$ is left adjoint (\cref{proposition: adjoint of for and inf}), we have
    \[
        \Ind\circ\For(F)\in \thcl{\Ind(G)}.
    \]
    Moreover, the functor $\Ind\circ\For$ is a splitting injection by \cite[Lemma 3.9.]{preprint_alexey_2015_on_equivariant_triangulated_categories}, 
    $F$ is a direct summand of $\Ind\circ\For(F)$.
    In particular, $F\in \thcl{\Ind(G)}$.

    The proof for the case of $\For$ is similar. Note that for this case, we used the assumption of $k=\CB$.
\end{proof}

The autoequivalences of equivariant categories are well studied in \cite{ploog_2007_equivariant_autoequivalences_for_finite_group_actions}.
We briefly review a part of this result that we will use.

\begin{prop}[{\cite[Lemma 5. (2) and Section 3.1]{ploog_2007_equivariant_autoequivalences_for_finite_group_actions}}]\label{proposition: commutativity of auto and for}
    Let $S$ be a smooth projective surface.
    Then, there is the following injection:
    \[
    \Auteq (S) \hookrightarrow\Auteq(\Hilb^n(S)), 
    \]
    such that any autoequivalence in the image of this embedding commutes with the forgetful functor $\For$:
    \[
    \begin{tikzcd}
D^b(S^n)_{\Sym(n)} \arrow[r] \arrow[d, "\For"] & D^b(S^n)_{\Sym(n)} \arrow[d, "\For"] \\
D^b(S^n) \arrow[r] & D^b(S^n). 
\end{tikzcd}
    \]
\end{prop}
The construction of the injection is given in \cref{setup: hilbert scheme part}.

\subsection{Hilbert Schemes of points on Surfaces}\label{subsection: hilb^n cat entro}
Throughout the rest of this section, we use the following setup and notation:
\begin{setup}\label{setup: hilbert scheme part}
    Let $S$ be a smooth surface, $S^n\coloneqq S\times\cdots\times S$, and $S^{[n]}$ the Hilbert scheme of $n$ points on $S$.
    For an equivalence $\Phi\in \Auteq S$, we denote by $\Phin$ the autoequivalence of $D^b(S^n)$ that acts on each component via $\Phi$.
    Because $\Phin$ commutes with the natural action of $n$-th symmetric group $\Sym(n)$ it descends as the autoequivalence of $D^b_{\Sym(n)}(S^n)$.
    
    In \cite{ploog_2007_equivariant_autoequivalences_for_finite_group_actions}, Ploog 
    showed that $\Phin$ descends to $D^b_{\Sym(n)}(S^n) \big(\cong D^b(S^{[n]})\big)$; we denote the induced functor by $\Psi$.
    By \cite{bridgeland_king_reid_2001_the_mckay_correspondence_as_an_equivalence_of_derived_categories} and the isospectral Hilbert scheme by Haiman \cite{haiman_2001_hilbert_schemes_polygraphs_and_the_macdonald_positivity_conjecture}, there is an equivalence
    \[
        \BKR: D^b_{\Sym(n)}(S^n) \stackrel{\cong}{\longrightarrow} D^b(S^{[n]}).
    \]
    Therefore, we regard $\Psi$ as an autoequivalence of $D^b(S^{[n]})$ when there is no risk of confusion.
\end{setup}
    By the remark below, the equivalence $\BKR$ does not change the spectral radius and categorical entropy.
\begin{rem}
Let us consider the following commutative diagram:
\begin{equation*}
\begin{tikzcd}
\DC \arrow[r, "\varphi"] \arrow[d, "\Phi"'] & \DC' \arrow[d, "\Psi"] \\
\DC \arrow[r, "\varphi"'] & \DC', 
\end{tikzcd}
\end{equation*}
where $\varphi$ is an equivalence.
Let $G$ be a generator of $\DC$ and $G'$ a generator of $\DC'$.
As $\varphi$ is an equivalence $\varphi^{-1}(G')$ (resp. $\varphi(G)$) is a generator of $\DC$ (resp. $\DC'$). Thus, we have

\begin{align*}
       h_t(\Psi) &=  \lim_{n\to \infty} \frac{1}{n}\log \big(\sum_{m\in\ZZ} \ext^{m}(G', \Psi^{n}(\varphi(G)))e^{-mt}\big)\\
       &= \lim_{n\to \infty} \frac{1}{n}\log \big(\sum_{m\in\ZZ} \ext^{m}(G', \varphi(\Phi^{n}(G)))e^{-mt}\big)\\
       &=\lim_{n\to \infty} \frac{1}{n}\log \big(\sum_{m\in\ZZ} \ext^{m}(\varphi^{-1}(G'), \Phi^{n}(G))e^{-mt}\big)\\
       &=h_t(\Phi).
\end{align*}
\begin{comment}
    Let $\Phi$ and $\Psi$ be autoequivalences of a triangulated category $\DC$.
    As the entropy is independent of the choice of generators, 

    \begin{align*}
       h_t(\Phi) &=  \lim_{n\to \infty} \frac{1}{n}\log \big(\sum_{m\in\ZZ} \ext^{m}(G, \Phi^{n}G')e^{-mt}\big)\\
       &= \lim_{n\to \infty} \frac{1}{n}\log \big(\sum_{m\in\ZZ} \ext^{m}(\Psi(G), \Psi\Phi^{n}(G'))e^{-mt}\big)\\
       &= \lim_{n\to \infty} \frac{1}{n}\log \big(\sum_{m\in\ZZ} \ext^{m}(\Psi(G), \Psi\Phi^{n}(\Psi^{-1}G'))e^{-mt}\big)\\
        &= \lim_{n\to \infty} \frac{1}{n}\log \big(\sum_{m\in\ZZ} \ext^{m}(\Psi(G), (\Psi\Phi\Psi^{-1})^{n}G'))e^{-mt}\big)\\
       &=h_t(\Psi\Phi\Psi^{-1})
    \end{align*}
    Thus, the conjugation by an equivalence does not change its categorical entropy.
\end{comment}
\end{rem}

\begin{prop}\label{proposition: hcat Psi = hcat n Phi}
    With the setting of \cref{setup: hilbert scheme part}, we have
    \[
    h_t(\Psi) = n\cdot h_t(\Phi).
    \]
    In particular, \(
    \hcat(\Psi) = n\hcat(\Phi).
    \)
\end{prop}

\begin{proof}    
    First, we show $h_t(\Phin) = n h_t(\Phi)$.
    Let $F$ and $F'\in D^b(S)$ be classical generators.
    Note that $F^{\boxn}$ and $F'^{\boxn}$ become classical generators of $D^n(S)$.
    The K\"unneth formula deduces that 
    \[
    \Ext^i\big(F^{\boxn}, (\Phin)^m((F')^{\boxn})\big) 
    \cong \bigoplus_{i_1 + \cdots i_n = i} \bigg(\bigotimes_{j=1}^{n} \Ext^{i_j}(F, \Phi^{m}(F'))\bigg).
    \]
    Therefore, we have
    \begin{align*}
        h_t(\Phin) 
        &= \lim_{m\to \infty} \frac{1}{m}\log \bigg(\sum_{k\in\ZZ} \ext^{k}\big(F^{\boxn}, (\Phin)^{m}((F')^{\boxn})\big)e^{-kt}\bigg)\\
        &=\lim_{m\to \infty} \frac{1}{m}\log \bigg( \sum_{k\in\ZZ} \big(\sum_{i_1 + \cdots +i_n = k} 
        \prod_{j=1}^{n} \ext^{i_j}(F, \Phi^{m}(F'))\big)e^{-kt}
        \bigg)\\
        &=\lim_{m\to \infty} \frac{1}{m}\log \bigg( \sum_{i_1,\ldots,i_n\in\ZZ} 
        \bigg(\prod_{j=1}^{n} \ext^{i_j}(F, \Phi^{m}(F'))e^{-i_jt}\bigg)
        \bigg)\\
        &=\lim_{m\to \infty} \frac{1}{m}\log \bigg(
        \sum_{k\in\ZZ} \ext^{k}(F, \Phi^m(F'))e^{-kt}
        \bigg)^{n}\\
        &=n\lim_{m\to \infty} \frac{1}{m}\log \big(
        \sum_{k\in\ZZ} \ext^{k}(F, \Phi^m(F'))e^{-kt}
        \big)\\
        &= n\cdot h_t(\Phi).
    \end{align*}

    Next, we show that the categorical entropies $h_t(\Psi)$ and $h_t(\Phin)$ coincide.
    Let $G$ and $G'$ be classical generators of $D^b(S^n)$.
    Using the independence of entropy from the choice of generators, the adjunction between $\For$ and $\Ind$, and \cref{lemma: For Inf preserves the classical generator}, we have the following equality:

    \begin{align*}
        h_t(\Phin) 
        &= \lim_{m\to \infty} \frac{1}{m}\log \big(\sum_{k\in\ZZ} \ext^{k}(G, (\Phin)^{m}(G'))e^{-kt}\big)\\
        &= \lim_{m\to \infty} \frac{1}{m}\log \big(\sum_{k\in\ZZ} \ext^{k}(G, (\Phin)^{m}\circ\For\circ\Ind (G'))e^{-kt}\big)\\
        &= \lim_{m\to \infty} \frac{1}{m}\log \big(\sum_{k\in\ZZ} \ext^{k}(G, \For\circ \Psi^{m}\circ\Ind (G'))e^{-kt}\big)\\
        &= \lim_{m\to \infty} \frac{1}{m}\log \big(\sum_{k\in\ZZ} \ext^{k}(\Ind(G), \Psi^{m}\circ\Ind (G'))e^{-kt}\big)\\
        &= h_t(\Psi).
    \end{align*}
    Combining these equations implies the statement.
    \end{proof}

    Summarizing the discussion of this section, we have the following result:
    \begin{thm}\label{theorem: positivity inherits to hilbert scheme}
    Let $S$ be a smooth projective surface that admits an autoequivalence with positive categorical entropy. Then,  $\Hilb^n(S)$ does so for any $n\in\ZZ_{>0}$.
    \end{thm}
    An analogous statement of \cref{proposition: hcat Psi = hcat n Phi} holds for the spectral radius of the induced linear action on the numerical Grothendieck group \(N(X)_\mathbb{C}\) under the additional assumption $q(S)=0$.
\begin{prop}\label{proposition: log rho Psi = n log rho Phi}
In \cref{setup: hilbert scheme part}, assume that $q(S) = 0$.
Then, we have
    \[
    \log \rho(\numg{\Psi}) = n\log\rho(\numg{\Phi}).
    \]
\end{prop}
\begin{comment}
\begin{proof}
    By the K\"unneth formula we have $N(S^n)_{\CB} = N(S)_{\CB}^{\otimes n}$. \memo{Wrong!!}
    As $\numg{(\Phin)} = (\numg{\Phi})^{\otimes n}$, we have
    \[
    \rho(\numg{(\Phin)}) = \big(\rho(\numg{\Phi}) \big)^n.
    \]
    Let $\lambda$ be an eigenvalue that achieves the spectral radius of $\numg{\Phi}$, i.e. $\abs{\lambda} = \rho(\numg{\Phi})$, and $v\in N(S)_{\CB}$ the eigenvector associated to the eigenvalue $\lambda$.
    Then, $v^{\otimes n}$ is an eigenvector associated to the eigenvalue $\lambda^n$ of $\numg{(\Phin)}$.
    As $v^{\otimes n}$ is $\Sym(n)$-invariant, $v^{\otimes n}\in N(S^n)_{\CB}^{\Sym(n)}$.
    Thus, we have
    \[
    \rho(\numg{\Psi}) \ge \abs{\lambda^n} = \rho(\numg{\Phi})^n.
    \]
    On the other hand, because the set of eigenvalues of $\numg{\Psi}$ is the subset of that of $\numg{(\Phin)}$ by \cref{proposition: commutativity of auto and for}, we have $\rho(\numg{\Psi}) \le\rho(\numg{\Phi})^n$.
    Thus, $\rho(\numg{\Psi}) = \rho(\numg{(\Phin)})$.

    The conclusion is 
    \[
    \log\rho(\numg{\Psi}) = \log\rho(\numg{(\Phin)}) = n\log\rho(\numg{\Phi}). 
    \]
\end{proof}
\end{comment}

\begin{proof}
Let $\Phi^H$ be the induced action on $H^*(S,\CB)$.
As $q(S)=0$, we have $H^{\mathrm{odd}}(S,\CB)=0$.
Via the Mukai vector, $N(S)_{\CB}$ is identified with
$H^0(S,\CB)\oplus \NS(S)_{\CB}\oplus H^4(S,\CB)$.
The orthogonal complement of this subspace is contained in $H^2(S,\CB)$.
As $\Phi^H$ preserves the grading by $p-q$ and the Mukai pairing, it also preserves this orthogonal complement.
By the Hodge index theorem, all eigenvalues of $\Phi^H$ on the orthogonal complement have absolute value one.
Since $\numg{\Phi}$ is an automorphism of $N(S)$, we have $\rho(\numg{\Phi})\ge 1$.
Thus,
\[
\rho(\Phi^H)=\rho(\numg{\Phi}).
\]

By \cite[Lemma A.7]{oberdieck_2025_lagrangian_planes_in_hyperkahler_varieties_of_k3ntype}, on the sector corresponding to a permutation of cycle type $(k_1,\ldots,k_l)$, the induced cohomological action is the exterior product of $\psi^{k_i}(\Phi^H)$, where each $\psi^{k_i}(\Phi^H)$ is conjugate to $\Phi^H$.
Thus, the spectral radius on this sector is $\rho(\Phi^H)^l\le \rho(\Phi^H)^n$.
By the cohomological BKR correspondence \cite[Lemma A.10]{oberdieck_2025_lagrangian_planes_in_hyperkahler_varieties_of_k3ntype}, we have
\[
\rho(\Psi^H)\le \rho(\Phi^H)^n=\rho(\numg{\Phi})^n.
\]
As the numerical action is induced from the cohomological action on the image of the Mukai vector, we have
\[
\rho(\numg{\Psi})\le \rho(\Psi^H)\le \rho(\numg{\Phi})^n.
\]

On the other hand, let $\lambda$ be an eigenvalue that achieves the spectral radius of $\numg{\Phi}$, i.e. $\abs{\lambda}=\rho(\numg{\Phi})$, and let $v\in N(S)_{\CB}$ be an eigenvector associated to $\lambda$.
Choose $w\in N(S)_{\CB}$ such that $\chi(v,w)\neq 0$.
Then $\Ind(v^{\boxtimes n})$ is an eigenvector of $\numg{\Psi}$ associated to the eigenvalue $\lambda^n$.
Moreover, by \cref{proposition: adjoint of for and inf},
\[
\chi_{\Sym(n)}\bigl(\Ind(v^{\boxtimes n}),\Ind(w^{\boxtimes n})\bigr)
=n!\chi(v,w)^n\neq 0.
\]
In particular, $\Ind(v^{\boxtimes n})$ is nonzero in the numerical Grothendieck group.
Thus, we have
\[
\rho(\numg{\Psi})\ge \abs{\lambda^n}=\rho(\numg{\Phi})^n.
\]
Combining the two inequalities, we obtain
\[
\rho(\numg{\Psi})=\rho(\numg{\Phi})^n.
\]
Therefore,
\[
\log\rho(\numg{\Psi})=n\log\rho(\numg{\Phi}).
\]
\end{proof}

\begin{thm}\label{theorem: GY property for Hilb}
    Let $S$ be a surface with $q(S)=0$ that fails to satisfy the property \ref{property: Gromov-Yomdin property}.
    Then, the Hilbert scheme of points $\Hilb^n(S)$ also does not satisfy \ref{property: Gromov-Yomdin property}.
\end{thm}
As mentioned above, surfaces with $(-2)$-curves, K3 surfaces, and Enriques surfaces do not satisfy the property (GY), while all varieties with ample or anti-ample canonical divisors, abelian surfaces, and bielliptic surfaces do.
\begin{proof}
    Let $\Phi$ be an autoequivalence of $D^b(S)$ that satisfies $\hcat(\Phi) > \log\rho(\numg{\Phi})$.
    Then, with the notation of \cref{setup: hilbert scheme part} with the assumption $q(S)=0$, by \cref{proposition: log rho Psi = n log rho Phi,proposition: hcat Psi = hcat n Phi} the following inequality holds:
    \begin{align*}
        \hcat(\Psi) = n\hcat(\Phi) > n\log\rho(\numg{\Phi}) = \log\rho(\numg{\Psi}).
    \end{align*}
    This is what we want.
\end{proof}

In particular, we have the following corollary for hyperk\"ahler manifolds of $K3^{[n]}$-type:

\begin{cor}
    Let $X$ be a hyperk\"ahler manifold of $K3^{[n]}$-type that is derived equivalent to $\Hilb^n(S)$ for a K3 surface $S$.
    Then, $X$ does not have the property \ref{property: Gromov-Yomdin property}.
\end{cor}
Note that if two hyperk\"ahler manifolds of $K3^{[n]}$-type are birational, then they are derived equivalent as shown in \cite{maulik_shen_yin_zhang_2025_the_dequivalence_conjecture_for_hyperkahler_varieties_via_hyperholomorphic_bundles}.

For a hyperk\"ahler manifold, the following section shows that there exists an explicit example that breaks the property (GY).

\section{Categorical Entropy on Hyperk\"ahler Manifolds}\label{section: categorical entropy on hyperkahler manifold}

In this section, we will show \cref{theorem in intro: GY property for HK}.

\begin{setup}\label{setup: def of Phi on hk mfd}
    Let $X$ be a projective hyperk\"ahler manifold of complex dimension $2n$ and $P$ be the $\PP{n}$-twist associated with the structure sheaf $\OO_X$.
    We take a very ample line bundle $\OO(1)$ on $X$ and $H\in\linsys{\OO(1)}$ a general member.

    Let us define an autoequivalence 
    \[
        \Phi\coloneqq P\circ (-\otimes \OO(-H)).
    \]

    Let $G\coloneqq \OO(1)\oplus\cdots\oplus\OO(2n+1)$ and 
    $G'\coloneqq \OO(-2n-1)\oplus\cdots\oplus\OO(-1)$ be generators of $D^b(X)$ and $k\in\ZZ_{>0}$, $l_m\in\ZZ_{>0}$ for each $m\in\ZZ_{>0}$.
\end{setup}

We use the following notation for ease of calculation:
\begin{itemize}
\item $d_i\coloneqq h^0(\OO(i))$ ($i\ge0$), and
\item $A_m(k, l_m)\coloneqq \Phi^m(\OO(-k))\otimes\OO(-l_m)$.
\end{itemize}
As $h^j(\OO(i)) = 0$ for $i, j>0$ by Kodaira vanishing, $d_i$ can be written as an $n$-th degree polynomial of $q_X(c_1(H))$, where $q_X$ is the (normalized) Beauville-Bogomolov-Fujiki form on $H^2(X, \ZZ)$ (see \cite[Corollary 23.18.]{book_gross_huybrechts_joyce_2003_calabiyau_manifolds_and_related_geometries}).
Note also that 
\[
A_{m+1}(k, l_{m+1}) = P(A_{m}(k, 1))\otimes\OO(-l_{m+1}).
\]

Recall that, for an object $E\in D^b(X)$, the $\PP{n}$-twist $P=P_{\OO_X}$ sends it to 
\begin{align*}
    P_{\OO_X}(E) &= 
    \Cone\big( \Cone(\RHom^{*-2}(\OO, E)\otimes\OO \to \RHom^*(\OO, E)\otimes\OO) \to E \big)\\
    &\cong \Cone\bigg( \Cone\big(\bigoplus_j \OO^{h^{j}(E)}[-j-2] \to \bigoplus_j \OO^{h^{j}(E)}[-j]\big) \to E \bigg).
\end{align*}

For an integer $k\in\ZZ_{>0}$, we define
\[
C^k_1\coloneqq \Cone\big(\RHom^{*-2}(\OO, \OO(-k-1))\otimes\OO\to \RHom^{*}(\OO, \OO(-k-1))\otimes\OO\big), 
\]
and
$C^k_m\coloneqq \Phi^{m-1}(C^k_1)$.
In addition, for any $m\in\ZZ_{>1}$, the object $D^k_{m+1}$ is defined to be 
\[
D^k_{m+1}\coloneqq \Cone\big(\RHom^{*-2}(\OO, C^k_{m}(-1))\otimes\OO\to \RHom^{*}(\OO, C^k_m(-1))\otimes\OO\big).
\]

By the description of $P$ above, we have the following relationship between them:
\begin{align}
    D^k_{m+1} \longrightarrow C^k_{m}(-1)\longrightarrow C^k_{m+1}\stackrel{[1]}{\longrightarrow}. 
    \label{triangle: induction on Ck}
\end{align}

On the other hand, the definition of $C^k_1$ deduces the following:
\begin{align}
        C^k_1 \longrightarrow \OO(-k-1) \longrightarrow \Phi(\OO(-k))\stackrel{[1]}{\longrightarrow}.
    \label{triangle: Ck1 and Phi OO(-k)}
\end{align}
Thus, by applying $\Phi^m$ to the triangle \eqref{triangle: Ck1 and Phi OO(-k)}, it holds that
\begin{align}
        C^k_{m+1} \longrightarrow \Phi^m(\OO(-k-1)) \longrightarrow\Phi^{m+1}(\OO(-k))\stackrel{[1]}{\longrightarrow}.\notag
\end{align}
In particular, we have
\begin{align}
    C^k_{m+1}\otimes\OO(-l_{m+1}) \longrightarrow A_m(k+1, l_{m+1}) \longrightarrow A_{m+1}(k, l_{m+1}) \stackrel{[1]}{\longrightarrow}.\label{triangle: C, A_m and A_m+1}
\end{align}

\begin{lem}\label{lemma: example of calculation m=1}
    The dimensions of cohomologies of $A_1(k, l_1) = P_{\OO}\big(\OO(-k)\otimes\OO(-1)\big)\otimes\OO(-l_1)$ are as follows:
    \begin{equation*}
        h^j(A_1(k, l_1)) = 
        \begin{cases}
            d_{k+l_{1}+1} & (j= 2n), \\
            d_{k+1}d_{l_1} & (j= 4n-1, 4n),\\
            0 & \text{otherwise.}
        \end{cases}
    \end{equation*}
\end{lem}
\begin{proof}
    By tensoring $\OO(-l_1)$ to the defining triangle of $C^k_1$, there is the triangle:
    \begin{align*}
        \OO(-l_1)^{\oplus d_{k+1}}[-2n-2] \longrightarrow \OO(-l_1)^{\oplus d_{k+1}}[-2n] \longrightarrow C^k_1(-l_1) \stackrel{[1]}{\longrightarrow}.
    \end{align*}
    This deduces that  
    \begin{equation*}
        h^j(C^k_1(-l_1)) =
        \begin{cases}
            d_{k+1}d_{l_1} &(j= 4n, 4n+1), \\
            0 & \text{otherwise}.
        \end{cases}
    \end{equation*}
    Therefore, the triangle \eqref{triangle: C, A_m and A_m+1} with $m=0$ is the form of
    \begin{align*}
        C^k_1(-l_1) \longrightarrow \OO(-k-1-l_1) \longrightarrow A_1(k, l_1) \stackrel{[1]}{\longrightarrow}, 
    \end{align*}
    and the cohomologies calculated above complete the proof.
\end{proof}

\begin{prop}\label{proposition: vanishing and top coh of A_m}
    For $j\ge 2n(m+1)+1$, we have
    \[
        h^j(A_m(k, l_m)) = 0
    \]
    for any $k$ and any $l_m$.
    Moreover, the dimension of the top cohomology of $A_m(k, l_m)$ is 
    \[
    h^{2n(m+1)}(A_m(k, l_m)) = d_{k+1}d_{l_m}d_1^{m-1}.
    \]
\end{prop}
\begin{proof}
    The case of $m=1$ is given in \cref{lemma: example of calculation m=1}.
    Suppose that for every $j \le m$ our claims hold.
    
    First, under the inductive assumption, we will show the following claim:
    \begin{claim}\label{claim: assumption on induction to vanishing of Ckm}
    For $j\ge 2n(m+1)+2$, we have
    \[
        h^j(C^k_m(-l_m)) = 0
    \]
    for any $k$ and any $l_m$.
    Moreover, the dimension of the top cohomology of $C^k_m(-l_m)$ is 
    \[
    h^{2n(m+1)+1}(C^k_m(-l_m)) = d_{k+1}d_{l_m}d_1^{m-1}.
    \]
    \end{claim}
    \begin{proof}[\proofname\  of \cref{claim: assumption on induction to vanishing of Ckm}]
        By the assumption, 
        for $j\ge 2nm+1$, $h^j\big(A_{m-1}(k+1, l_{m-1})\big) = 0$ and 
        for $i\ge2n(m+1)+1$, $h^i\big(A_m(k, l_{m-1})\big) = 0$. 
        Thus, the triangle \eqref{triangle: C, A_m and A_m+1} shows the vanishing of the cohomology of $C^k_m(-l_{m})$.
        In particular, the long exact sequence of cohomologies deduces that
        \[
            H^{2n(m+1)+1}(C^k_m(-l_{m}))\cong H^{2n(m+1)}(A_m(k, l_m))
        \]
        and thus, again by the inductive assumption, the claim holds.
    \end{proof}
    Now, let us go back to the proof of \cref{proposition: vanishing and top coh of A_m}.

    Taking long exact sequence of the defining triangle of $D^k_{m+1}$ tensored with $\OO(-l_{m+1})$
    \begin{align}
        \RHom^{*-2}(\OO, C^k_{m}(-1))\otimes\OO(-l_{m+1})&\to \RHom^{*}(\OO, C^k_m(-1))\otimes\OO(-l_{m+1}) \notag \\
        &\longrightarrow D^k_{m+1}(-l_{m+1})\stackrel{[1]}{\longrightarrow}
        \notag
    \end{align}
    together with \cref{claim: assumption on induction to vanishing of Ckm} yields the following cohomology table (\cref{table: cohomology of D^k_m+1}):
{\small
    \begin{table}[ht]
        \centering
        \begin{tabular}{c|c|c|c}
            &$\RHom^{*-2}(\OO, C^k_{m}(-1))\otimes\OO(-l_{m+1})$ & $\RHom^{*}(\OO, C^k_{m}(-1))\otimes\OO(-l_{m+1})$ & $D^k_{m+1}(-l_{m+1})$ \\ \hline
            $\vdots$ & $\vdots$ & $\vdots$ & $\vdots$ \\ 
            $2n(m+2)$ & ----- & ----- & -----\\
            $2n(m+2)+1$ & ----- & $d_{k+1}d_{l_{m+1}}d_1^{m}$ & -----\\
            $2n(m+2)+2$ & ----- & $0$ & $d_{k+1}d_{l_{m+1}}d_1^{m}$\\
            $2n(m+2)+3$ & $d_{k+1}d_{l_{m+1}}d_1^{m}$ & $0$ & $0$ \\
            $2n(m+2)+4$ & $0$ & $0$ & $0$ \\
            $\vdots$ & $\vdots$ & $\vdots$ & $\vdots$ 
        \end{tabular}
        \vspace{1cm}
        \caption{Dimension of the Top Cohomology of $D^k_{m+1}$}
        \label{table: cohomology of D^k_m+1}
    \end{table}
}
    
    Note that here we use 
    \[
    h^{2n(m+1)+1}(C^k_m(-1)) = d_{k+1}d_{1}d_1^{m-1} =d_{k+1}d_1^{m}
    \]
    by \cref{claim: assumption on induction to vanishing of Ckm}.
    
    From the triangle \eqref{triangle: induction on Ck} and \cref{claim: assumption on induction to vanishing of Ckm}, a calculation similar to \cref{table: cohomology of D^k_m+1} shows the isomorphism of cohomologies
    \[
    H^{2n(m+2)+2}(D^k_{m+1}(-l_{m+1})) \cong H^{2n(m+2)+1}(C^k_{m+1}(-l_{m+1}))
    \]
    and the vanishing
    \[
    H^{j}(C^k_{m+1}(-l_{m+1}))=0 \quad \text{for} \quad j\ge2n(m+2)+2.
    \]
    Combining a similar calculation above to \eqref{triangle: C, A_m and A_m+1} again and our inductive assumption leads to 
    \[
    H^{2n(m+2)+1}(C^k_{m+1}(-l_{m+1}))\cong H^{2n(m+2)}(A_{m+1}(k, l_{m+1}))
    \]
    and the vanishing
    \[
    H^{j}(A_{m+1}(k, l_{m+1}))=0 \quad \text{for} \quad j\ge2n(m+2)+1.
    \]
    Hence, it completes the proof.
\end{proof}

\begin{comment}
\begin{lem}\label{lemma: chi of Am is positive}
    For every $m\in\ZZ_{\ge0}$, 
    \[
    \chi(A_m) >0.
    \]
\end{lem}
\begin{proof}
    Because the $\PP{n}$-twist $P$ acts on the cohomology ring $H^{*}(X, \CB)$ as an identy (\cite[Remark 2.4.]{huybrechts_thomas_2006_bbb_pobjects_and_autoequivalences_of_derived_categories}), 
    $\ch(B_m) = \ch(\OO(-(m+1)))$ holds.
    Thus, 
    \begin{align*}
        \chi(A_m) &= \chi(B_m\otimes\OO(-1))\\
        &= \chi(\OO(-m-2))\\
        &= h^0(\OO(m+2)) > 0.
    \end{align*}
\end{proof}
\end{comment}

\begin{prop}\label{proposition: positivity of cat entr of Phi}
    In \cref{setup: def of Phi on hk mfd}, we have 
    \[
        \hcat(\Phi) > 0.
    \]
\end{prop}
\begin{proof}
    By the calculation of cohomologies in \cref{proposition: vanishing and top coh of A_m}, 
    the following lower bound for $\delta'(G, \Phi^m(G'))$ holds:
    \begin{align*}
        \delta'(G, \Phi^m(G')) &= \sum_{k=1}^{2n+1}\sum_{l_m=1}^{2n+1} \sum_{i\in\ZZ} \ext^i(\OO(l_m), \Phi^m(\OO(-k)))\\
        &= \sum_{k=1}^{2n+1}\sum_{l_m=1}^{2n+1} \sum_{i\in\ZZ} h^i(A_m(k, l_m))\\
        &\ge \sum_{k=1}^{2n+1}\sum_{l_m=1}^{2n+1} d_{k+1}d_{l_m}d_1^{m-1}\\
        &\ge d_1^{m+1}.
    \end{align*}
    Therefore, we obtain 
    \begin{align*}
        \hcat(\Phi) 
        &= \lim_{m\to \infty} \frac{1}{m}\log \big(\sum_{k\in\ZZ} \ext^{k}(G, (\Phi)^{m}(G'))\big)\\
        &\ge \lim_{m\to \infty} \frac{1}{m}\log \big(d_1^{m+1}\big)\\
        &= \log d_1 > 0.
    \end{align*}
\end{proof}

\begin{thm}\label{theorem: GY property of hyperkahler manifold}
    Any hyperk\"ahler manifold $X$ admits an autoequivalence with positive categorical entropy.
    Moreover, $X$ does not satisfy the property \ref{property: Gromov-Yomdin property}.
\end{thm}
\begin{proof}
    \cref{proposition: positivity of cat entr of Phi} shows the first assertion.
    For the latter one, it is enough to show that the logarithm of the spectral radius of $\Phi$ vanishes.
    Note that while the tensoring with $\OO(-H)$ acts non-trivially on the cohomology ring via multiplication by the Chern character $\ch(\OO(-H)) = e^{-H}$, this action is unipotent (its eigenvalues are all 1). 
    Combined with the fact that the $\PP{n}$-twist $P$ acts as the identity on cohomology (\cite[Remark 2.4.]{huybrechts_thomas_2006_bbb_pobjects_and_autoequivalences_of_derived_categories}), the induced linear map $\numg{\Phi}$ has spectral radius $\rho(\Phi^N) = 1$ and thus we have 
    \[
    \log\rho(\numg{\Phi}) = 0.
    \]
    Thus, we have the inequality
    \[
        \hcat(\Phi) >0 = \log\rho(\numg{\Phi}).
    \]
    As a result, the autoequivalence $\Phi$ is also an example that does not satisfy the Gromov-Yomdin type equality.
\end{proof}

By replacing the very ample divisor by a sufficiently high multiple, or simply
by taking powers and using $\hcat(\Phi^m)=m\hcat(\Phi)$, we can construct
autoequivalences with arbitrarily large categorical entropy and numerical
spectral radius one.

\begin{cor}
\label{corollary: hyperkaheler the set of hcat is unbounded}
Let $X$ be a projective hyperk\"ahler manifold.
Then, for any $M>0$, there exists an autoequivalence $\Phi\in\Auteq X$ such that $\rho(\numg{\Phi}) = 1$ and $\hcat(\Phi)>M$. Equivalently, we have
\[
\sup\{\hcat(\Phi) \mid \Phi\in\Auteq X, \rho(\numg{\Phi}) = 1\} =\infty.
\]
\end{cor}

\section{Enriques Manifolds}\label{section: Enriques manifolds}
In this section, we will see the proof of \cref{theorem in intro: GY property for Enriques manifolds}.
First, let us recall the definition of \emph{Enriques manifolds} introduced by Oguiso and Schr\"oer. 
\begin{dfn}[\cite{oguiso_schroer_2011_enriques_manifolds}]\label{definition: Enriques manifolds}
    A connected complex manifold $Y$ is called an \emph{Enriques manifold} if it is not simply connected and its universal cover $X$ is a hyperk\"ahler manifold.
\end{dfn}

Note that \cite[Proposition 2.4.]{oguiso_schroer_2011_enriques_manifolds} says that $\pi_1(Y)$ is a cyclic group. Thus, the universal cover $X$ is the canonical cover of $Y$.

\begin{prop}[{\cite[Corollary 2.7.]{oguiso_schroer_2011_enriques_manifolds}}]
    Every Enriques manifold $Y$ is projective. 
    Moreover, its universal covering $X$ is also projective.
\end{prop}
This ensures the existence of an autoequivalence $\Phi\in\Auteq X$ that is constructed in \cref{section: categorical entropy on hyperkahler manifold} and has positive categorical entropy.

Throughout this section, we follow the following notation and settings.
\begin{itemize}
    \item Let $Y$ be an Enriques manifold of dimension $2n$ and $\pi: X\to Y$ be its universal covering.
    \item Fix a very ample line bundle $\OO_X(1)$ on $X$.
    \item Let $\Phi\coloneqq P_{\OO_X} \circ (-\otimes\OO(-1))$. 
    \item Let $G \coloneqq \Deck(\pi) (\cong\pi_1(Y))$ be the group of deck covering transformations. Choose $g\in G$ as a generator of the cyclic group $G$.
\end{itemize}

In the rest of this section, we closely follow the discussion in \cite[Section 4]{yoshida_2026_a_note_on_categorical_entropy_of_bielliptic_surfaces_and_enriques_surfaces}, which discusses the case of surfaces.

Before proving \cref{theorem in intro: GY property for Enriques manifolds}, 
we show that $\Phi$ descends to $D^b(Y)$, i.e., that there exists an autoequivalence $\Psi\in\Auteq(Y)$ whose lift is $\Phi$.

\begin{prop}
    Let $g^*$ be the left derived pullback of the deck transformation $g$. 
    Then, it holds that
    \[
    P_{\OO_X}\circ g^{*} \cong g^{*}\circ P_{\OO_X}. 
    \]
\end{prop}
\begin{proof}
    For any autoequivalence $\phi\in\Auteq X$,
    \[
     \phi\circ P_{\OO_X}\cong P_{\phi(\OO_X)}\circ \phi
    \]
    follows from \cite[Lemma 2.4.]{krug_2015_on_derived_autoequivalences_of_hilbert_schemes_and_generalized_kummer_varieties}.
    Hence, because $g^*\OO_X\cong\OO_X$, we have the assertion.
\end{proof}

In particular, if we choose $\OO(1)$ as $g^*$-invariant we may assume 
\[
\Phi\circ g^{*} \cong g^{*}\circ \Phi.
\]
Thus, there exists an autoequivalence $\Psi\in\Auteq Y$ such that $\Phi$ is a lift of it by \cite[Section 3.3.]{ploog_2007_equivariant_autoequivalences_for_finite_group_actions}.

The converse is also true.
Indeed, in \cite{bridgeland_maciocia_2001_complex_surfaces_with_equivalent_derived_categories}, Bridgeland and Maciocia showed that any autoequivalence of a derived category of a smooth variety with finite canonical order lifts to that of its canonical covering.
In this situation, it is known that the categorical entropy of one is completely determined by that of the other.
\begin{prop}[{\cite[Proposition 4.2.]{yoshida_2026_a_note_on_categorical_entropy_of_bielliptic_surfaces_and_enriques_surfaces}}]
\label{proposition: entropy coincide for descent and lift}
    Let $Y$ be a smooth projective variety with a torsion canonical sheaf and 
    $\psi\in\Auteq Y$. 
    Let $\phi\in \Auteq X$ be a lift of $\psi$.
    Then, the equality of entropies holds, i.e., 
    \[
        h_t(\phi) = h_t(\psi).
    \]
\end{prop}

\begin{lem}\label{lemma: spectral radius of Psi on Enriques mfd}
    With the settings above, the logarithm of the spectral radius of $\Psi$ vanishes, i.e., 
    \[
    \log\rho(\numg{\Psi}) = 0.
    \]
\end{lem}
\begin{proof}
    The covering map $\pi$ gives an injection between the numerical Grothendieck groups 
    $\pi^*:N(Y)_{\RB} \hookrightarrow N(X)_{\RB}$.
    Moreover, the spectral radius of $\numg{\Psi}$ can be bounded above by that of $\numg{\Phi}$ from the commutativity of actions on the numerical Grothendieck groups of $\Phi$ and $\pi^*$.
    Thus, we have
    \[
    0 = \log\rho(\numg{\Phi}) \ge \log\rho(\numg{\Psi}) \ge 0.
    \]
\end{proof}

Combining \cref{proposition: entropy coincide for descent and lift} and \cref{lemma: spectral radius of Psi on Enriques mfd}, we have the following theorem:

\begin{thm}\label{theorem: pos entr and counteg of GY for Enriques manifolds}
    Let $Y$ be an Enriques manifold.
    Then, $Y$ admits an autoequivalence $\Psi\in \Auteq Y$ that has positive categorical entropy and $\log\rho(\numg{\Psi})=0$. In particular, $Y$ fails to satisfy the \ref{property: Gromov-Yomdin property} property.
\end{thm}

By the same argument as in \cref{corollary: hyperkaheler the set of hcat is unbounded}, we obtain the unboundedness of the set of categorical entropies for Enriques manifolds.

\begin{cor}
\label{corollary: Enriques the set of hcat is unbounded}
Let $Y$ be an Enriques manifold.
Then, for any $M>0$, there exists an autoequivalence $\Phi\in\Auteq Y$ such that $\rho(\numg{\Phi}) = 1$ and $\hcat(\Phi)>M$. Equivalently, we have
\[
\sup\{\hcat(\Phi) \mid \Phi\in\Auteq Y, \rho(\numg{\Phi}) = 1\} =\infty.
\]
\end{cor}